\documentclass[12pt]{article}

\usepackage{fullpage}
\usepackage{amssymb}
\usepackage{amsmath}
\usepackage{amsthm}
\usepackage{url}
\usepackage{hyperref}

\newtheorem{theorem}{Theorem}

\newtheorem{conjecture}{Conjecture}
\newtheorem{definition}{Definition}

\begin{document}

\title{On a conjecture regarding primality of numbers constructed from prepending and appending identical digits}

\author{Chai Wah Wu\\ IBM T. J. Watson Research Center\\ P. O. Box 218, Yorktown Heights, New York 10598, USA\\e-mail: chaiwahwu@member.ams.org}
\maketitle

\begin{abstract}
Consider the operation of adding the same number of identical digits to the left and to the right of a number $n$.  In OEIS sequence A090287, it was conjectured that this operation will not produce a prime if and only if $n$ is a palindrome with an even number of digits.  We show that this conjecture is false by showing that this property also holds for $n=231$, $n=420$, and an infinite number of other values of $n$.  The analysis involves looking at the prime factors of repunits and we present an algorithm to find $n$ which do not produce a prime under this operation.
\end{abstract}

\section{Introduction}

The topic of this paper concerns the primality of numbers obtained by adding digits to the left and/or to the right of a number.  There has been many papers and OEIS sequences \cite{oeis} denoted to this subject.  For instance, Ref. \cite{angell:trunc:1977} considers right and left truncatable primes, i.e. primes that remained primes when successive digits in base $b$ are removed from the left and from the right respectively (see also OEIS sequences A024785 and A133758).   In Ref. \cite{kahan:trunc:1998} the authors consider restricted left truncatable primes which are left truncatable primes that cannot be obtained by left truncation of another prime.  In Ref. \cite{honaker:pyramid:2000} palindromic primes are studied that remain palindromic primes when one (or more) digit is removed from both the left and the right (see also OEIS sequences A256957 and A034276). 
In this paper we study another such problem.  OEIS sequence A090287 (\url{http://oeis.org/A090287}) is defined as follows: $a(n)$ is the smallest prime obtained by inserting $n$ between two copies of a number with identical digits, or $0$ if no such prime exists.  In other words, $a(n)$ (when it is nonzero) is the smallest prime formed by concatenating (in decimal) $u$, $n$ and $u$ where all the digits of $u$ are identical. The author of the sequence conjectured that 
$a(n) = 0$ if and only if $n$ is a palindrome with an even number of digits.  We show that this conjecture is false.  In particular, we show that there are an infinite number of counterexamples, with the first one being $n=231$.

\section{Additional values of $n$ for which $a(n) = 0$} \label{sec:additional}
\begin{definition}
Define a {\em repunit} as a number of the form $R(n) = \frac{10^n-1}{9}$.  It is represented in decimal by $n$ 1's.
\end{definition}
\begin{definition} \label{def:f-10}
Define $f(d,m,n)$ as the number obtained by concatenating $m$ times the digit $d$ in front and after the number $n$, i.e.
$f(2,3,45) = 22245222$. In other words, $f(d,m,n) = dR(m)10^{t+m}+n10^m+dR(m)$ is represented as a concatenation of $dR(m)$, $n$ and $dR(m)$ where $t$ is the number of digits of $n$.
\end{definition}
Then the conjecture in \url{http://oeis.org/A090287} states that
\begin{conjecture}\label{conj:a090287}
$f(d,m,n)$ is composite for all integers $m$, $d$ with $0 < d < 10$ and $m>0$ if and only if $n$ is a palindrome with an even number of digits.
\end{conjecture}

One direction of the conjecture is easy to prove.  If $n$ is a palindrome with an even number of digits, then $f(d,m,n)$ is a palindrome with an even number of digits.  The sum of the odd-numbered digits is the same as the sum of the even-numbered digits and the well-known test to determine divisibility by $11$ shows that $f(d,m,n)$ is a multiple of $11$.  Since $f(d,m,n) > 11$, it must be composite.

We now show that the other direction is false by giving a counterexample.

\begin{theorem}\label{thm:231}
$f(d,m,231)$ is composite for all $0 < d < 10$ and $m \geq 0$.  
\end{theorem}
\begin{proof}
We will make frequent use of the following simple observation.  If $n$ is a concatenation of numbers $m$ and $k$, then
$p|n$ if $p|m$ and $p|k$.  This is due to the fact that $n = m10^i + k$ for some integer $i$.  Next note that $f(d,m,n)$ is a concatenation of multiple copies of $d$ and $n$.

Clearly if $d$ is even, then $f(d,m,231)$ is even and since $f(d,m,231) > 2$, it is composite.
Note that $231 = 3\cdot 7 \cdot 11$.  Therefore the numers $f(3,m,231) > 3$ and $f(9,m,231) > 9$ are divisible by $3$ by the observation above and thus composite.  Similarly the number $f(7,m,231) > 7$ is divisible by $7$.  $f(5,m,231) > 5$ has $5$ as the last digit and is thus divisible by $5$.  As for $f(1,m,231)$ we consider $2$ cases.  If $m$ is even, then $R(m)$ is divisible by $11$ (as it is a concatenation of multiple $11$'s).  Since $231$ is also divisible by $11$, $f(1,m,231)$ is divisible by $11$.

If $m$ is odd, then $f(1,m,231)$ is a concatenation of (possibly multiple copies of) $111111$ and one of the following numbers:
$12311$, $111231111$ and $1111123111111$.
The factorization of these numbers are: 
\begin{itemize}
\item $111111 = 3\cdot 7\cdot 11\cdot 13\cdot 37$
\item $12311 = 13 \cdot 947$
\item $111231111 = 3\cdot 19 \cdot 193 \cdot 10111$
\item $1111123111111 = 13\cdot 8231 \cdot 10384037$
\end{itemize}
Thus $111111$ is divisible by both $3$ and $13$.  The numbers $12311$, $111231111$ and $1111123111111$ are divisible by either
$3$ or $13$. By the observation above this implies that for odd $m$ $f(1,m,231)$ is divisible by either $3$ or $13$.  Since
$f(1,m,231)$ has at least $3$ digits, it is composite.
\end{proof}

\begin{theorem}
$f(d,m,420)$ is composite for all $0 < d < 10$ and $m \geq 0$.  
\end{theorem}
\begin{proof}
Since $420 = 2^2\cdot 3\cdot 5 \cdot 7$, similar to the proof of Theorem \ref{thm:231}, $f(d,m,420)$ is composite for all $m \geq 0$ and $1 < d < 10$.
$f(1,m,420)$ is a concatenation of (possibly multiple copies of) $111111$ and one of the following numbers:
$420$, $14201$, $1142011$, $111420111$, $11114201111$, $1111142011111$.
The factorizations are:
\begin{itemize}
\item $14201 = 11\cdot 1291$
\item $1142011 = 13\cdot 107 \cdot 821$
\item $111420111 = 3\cdot 11 \cdot 317 \cdot 10651$
\item $11114201111 = 13\cdot 179 \cdot 293 \cdot 16301$
\item $1111142011111 = 11\cdot 31\cdot 337 \cdot 9669083$
\end{itemize}
This implies that $f(1,m,420)$ is divisible by either $3$, $11$ or $13$.
\end{proof}

\begin{theorem}
$f(d,m,759)$ is composite for all $0 < d < 10$ and $m \geq 0$.  
\end{theorem}
\begin{proof}
Since $759 = 3\cdot 11 \cdot 23$, similar to the proofs above, we only need to check $d=1$ and $d=7$.  Again, since $759$ is divisible by $11$, we only need to check for odd $m$.
$f(1,m,759)$ is a concatenation of (possibly multiple copies of) $111111$ and one of the following numbers:
$17591$, $111759111$, $1111175911111$.
The factorizations are:
\begin{itemize}
\item $17591 = 7^2\cdot 359$
\item $111759111 = 3^2\cdot 127 \cdot 97777$
\item $1111175911111 = 7\cdot 12373 \cdot 12829501$
\end{itemize}
This implies that $f(1,m,420)$ is divisible by either $3$ or $7$.
As for $f(7,m,759)$, it is a concatenation of (possibly multiple copies of) $777777$ and one of the following numbers:
$77597$, $777759777$, $7777775977777$.
The factorizations are:
\begin{itemize}
\item $77597 = 13\cdot 47\cdot 127$
\item $777759777 = 3^2\cdot 151 \cdot 572303$
\item $7777775977777 = 13\cdot 1289\cdot 5189\cdot 89449$
\end{itemize}
Thus $f(7,m,759)$ is divisible by either $3$ or $13$.
\end{proof}

\begin{theorem}
$f(d,m,6363)$ is composite for all $0 < d < 10$ and $m \geq 0$.  
\end{theorem}
\begin{proof}
Since $6363 = 3^2\cdot 7 \cdot 101$, similar to the proofs above, we only need to check $d=1$.  Since $1111=11\cdot 101$,
$f(1,m,6363)$ is divisible by $101$ for $m$ a multiple of $4$.  For $m$ not a multiple of $4$,  
$f(1,m,6363)$ is a concatenation of (possibly multiple copies of) $R(12) = 3\cdot 7\cdot 11\cdot 13\cdot 37\cdot 101\cdot 9901$ and $f(1,m,6363)$ for some $m \in \{0,1,2,3,5,6,7,9,10,11\}$.  Looking at the factorizations of these numbers shows that
for all $m\geq 0$, $f(1,m,6363)$ is divisible by one of the following prime factors: $3$, $7$, $13$, $101$.  
\end{proof}

\begin{theorem}
$f(d,m,10815)$ is composite for all $0 < d < 10$ and $m \geq 0$.  
\end{theorem}
\begin{proof}
Since $10185 = 3\cdot 5\cdot 7 \cdot 103$, similar to the proof of Theorem \ref{thm:231}, the number $f(d,m,10815)$ is composite for all $m \geq 0$ and $1 < d < 10$.
$f(1,m,231)$ is a concatenation of (possibly multiple copies of) $111111$ and one of the following numbers:
$10815$, $1108151$, $111081511$, $11110815111$, $1111108151111$, $111111081511111$.
The factorizations are:
\begin{itemize}
\item $1108151 = 11\cdot 100741$
\item $111081511 = 37\cdot 67 \cdot 44809$
\item $11110815111 = 3 \cdot 11\cdot 157 \cdot 607 \cdot 3533$
\item $1111108151111 = 7\cdot 158729735873$
\item $111111081511111 = 11\cdot 37 \cdot 661\cdot 413010893$
\end{itemize}
This implies that $f(1,m,10815)$ is divisible by either $7$, $11$ or $37$.
\end{proof}

\begin{theorem} \label{thm:composite}
If $n$ has an even number of digits and $n$ is divisible by $11$, then $f(d,m,n)$ is composite for all $0 \leq d < 10$ and $m > 0$. 
\end{theorem}
\begin{proof}
This can be proved with the same argument as the one above that was used to prove one direction of Conjecture \ref{conj:a090287}, but let us prove this with the notation we have defined so far. Similar to the proofs above, we only need to check $d=1$, $d=3$ and $d=7$. Let $q$ be the number of digits of $n$.
Define $n_2$ as the concatenation of $d$, $n$ and $d$, i.e. $n_2 = d10^{q+1}+10n+d$.
$f(d,m,n)$ is a concatenation of (possibly multiple copies of) $11d$ (i.e. `$dd$') and either $n$ or $n_2$.  The divisibility test for $11$ shows that for even $q$, the number $d10^{q+1}+d$ is divisible by $11$.  
This implies that $n_2$ is divisible by $11$ and thus $f(d,m,n)$ is divisible by $11$.
\end{proof}

The reasoning in the results above allows us to derive the following test to provide a sufficient condition for when $f(d,m,n)$ for fixed $n$, $d$ is composite for all $m > 1$:  

\begin{enumerate}
\item Pick $1< k \leq k_{\max}$.   
\item For each $0 \leq i < k$, let $w_i$ be the concatenation of $dR(i)$, $n$ and $dR(i)$.  
\item If $\mbox{gcd}(dR(k), w_i) > 1$ for each $i$, then $f(d,m,n)$ is composite for all $m>1$.
\end{enumerate}
with the convention that the concatenation of $dR(0)$, $n$ and $dR(0)$ is equal to $n$. Implementing this test in a computer program, we found the following values of $n$ which have an odd number of digits or are not divisible by $11$ such that $a(n) = 0$: $231$, $420$, $759$, $2814$, $6363$, $9177$, $10815$, $12663$, $15666$, $18669$, $19362$, $21672$, $24675$, ....
We found $4919$ such numbers for $n\leq 10^7$.  In addition, we found that all of these $4919$ numbers are multiples of $3$ and satisfy the test for all $d$ (that need to be checked) with $k=6$ except for $6363$, $488649$, $753774$ which needed $k=12$, the numbers $921333$ and $8872668$ which needed $k=8$ and $5391498$ which needed $k=30$. The minimal $k$ for each digit $d$ may depend on $d$, but it is clear that the least common multiple of all these $k$'s will work for all $d$.

Note that the proof in Theorem \ref{thm:231} relies on $f(1,m,231)$ sharing a prime factor with some repunit $R(k)$ for $0\leq m < k$. In this case it is true for $R(6) = 111111$.  This may suggest that we focus on repunits with many prime factors such as $R(k)$ for $k=6, 12, 15, 16, 18, 20$, etc.

\begin{conjecture}\label{conj:a}
If $n$ is not a multiple of $11$ with an even number of digits and $a(n) = 0$, then $n$ is a multiple of $3$.
\end{conjecture}

There are other values of $n$ such as $366$, $1407$ for which $a(n)$ is still unknown. An interesting case is 
$a(1414) = f(3,1207,1414)$ which has presumably\footnote{We use the word ''presumably'' since we use a probabilistic primality test which even though it guarantees the number is composite when the test determines it to be so, there is a very small chance that a number determined by the test to be prime can be composite.} $2418$ digits.

\section{Prepending identical digits and appending identical digits}

Consider the following variations of the sequence $a(n)$. 
\begin{definition}
$b(n)$ is the smallest prime obtained by appending $n$ to a number with identical digits, or $0$ if no such prime exists (\url{https://oeis.org/A256480}).
$c(n)$ is the smallest prime obtained by appending a number with identical digits to $n$, or $0$ if no such prime exists (\url{https://oeis.org/A256481}).
\end{definition}

Clearly $b(n) = 0$ if $n$ is even or divisible by $5$.  Furthermore $c(n)$ coincides with OEIS A030665 (\url{https://oeis.org/A030665}) for $n < 20$.
The sufficient condition for testing whether $a(n) = 0$ described above has corresponding versions for $b(n)$ and $c(n)$.  For $b(n)$ we have:

\begin{enumerate}
\item Pick $1\leq d \leq 9$ such that the prime factors of $d$ do not divide $n$.
\item Pick $1< k \leq k_{\max}$.   
\item For each $0 \leq i < k$, let $w_i$ be the concatenation of $dR(i)$ and $n$.  
\item If for each $d$, $\mbox{gcd}(dR(k), w_i) > 1$ for each $i$, then $b(n) = 0$.
\end{enumerate}

For $c(n)$ we have:

\begin{enumerate}
\item Pick $d$ from the set $\{1, 3, 7, 9\}$ such that the prime factors of $d$ do not divide $n$.
\item Pick $1< k \leq k_{\max}$.   
\item For each $0 \leq i < k$, let $w_i$ be the concatenation of $n$ and $dR(i)$.  
\item If for each $d$, $\mbox{gcd}(dR(k), w_i) > 1$ for each $i$, then $c(n) = 0$.
\end{enumerate}

\begin{conjecture} \label{conj:b}
$b(n) = 0$ if and only if $n$ is even or $n$ is divisible by $5$.
\end{conjecture}

\begin{conjecture} \label{conj:c}
If $c(n) = 0$, then $n$ is divisible by $3$.
\end{conjecture}

The decimal value of $c(6069)$ has presumably $1529$ digits.
Furthermore, for $n \leq 15392$, it appears\footnote{The probabilistic primality test has shown that $c(n) > 0$ for $n\neq 6930$ and $n\leq 15392$.  Even though we have proven that $c(6930) = 0$, the cases where $c(n) > 0$ is highly likely to be correct, but not with certainty.} that $c(n) = 0$ if and only if $n = 6930$. 

Probabilistic primality tests also suggest that Conjecture \ref{conj:a},  Conjecture \ref{conj:b}, Conjecture \ref{conj:c} are true for $n\leq 37443$, $n \leq 10^7$ and $n\leq 10^6$ respectively.  

\begin{theorem}
$c(n) = 0$ for $n \in \{6930$, $50358$, $56574$, $72975$, $76098$, $79662$, $82104$, $118041$, $160920 \}$
\end{theorem}
\begin{proof}
Consider the case $n=6930$.
Since $6930=2\cdot 3^2\cdot 5\cdot 7\cdot 11$, similar to the proofs above we only need to check the primality of appending a repunit $R(m)$ to $6930$.  Again we only need to check for odd $m$ since $11$ divides $6930$.  The factorizations
\begin{itemize}
\item $69301 = 37\cdot 1873$ 
\item $6930111 = 3 \cdot 109 \cdot 21193$
\item $693011111 = 13 \cdot 19 \cdot 2805713$
\end{itemize}
show that $6930$ appended with $R(m)$ all shared a factor with $111111$ and is divisible by either $3$, $11$, $13$ or $37$.  The other cases are similar, with the exception that some of these numbers (such as $56574$ and $72975$) do not divide $11$ and thus we need to consider the factorizations of $n$ appended with $R(m)$ for $0\leq m \leq 5$ and show that they all share a nontrivial factor with $111111$.
\end{proof} 

In particular, we found $67$ numbers less than $10^6$ such that $c(n) = 0$ and for all these numbers this is shown by appending $n$ with $R(m)$ for $0\leq m \leq 5$ and showing that they all share a nontrivial factor with $111111$.

\section{Other bases}
Even though the analysis above focuses on numbers and their decimal representations, it is clear that this extends readily to other number bases as well.

\begin{definition}
Define a {\em repunit} in base $b$ as a number of the form $R_b(n) = \frac{b^n-1}{b-1}$.  It is represented in $b$-ary representation by $n$ 1's.
\end{definition}
\begin{definition} \label{def:fb}
For a number base $b > 2$ and $d<b$, define $f_b(d,m,n)$ as the number obtained by concatenating $m$ times the digit $d$ in front and after the number $n$ in base $b$.
In other words, $f_b(d,m,n) = dR_b(m)b^{t+m}+nb^m+dR_b(m)$ is represented as a concatenation in base $b$ of $dR_b(m)$, $n$ and $dR_b(m)$ where $t$ is the number of digits of $n$ in base $b$.
\end{definition}

For example, since $34 = 100010_2$, $f_2(1,3,34) = 111100010111_2 = 3863$. 
The proof of Theorem \ref{thm:composite} can be used with minor modification to prove:
\begin{theorem} \label{thm:baseb}
If $n$ has an even number of digits in base $b$ and $n$ and $b+1$ are not coprime (i.e. $\mbox{gcd}(n,b+1) > 1$), then $f_b(d,m,n)$ is composite for all $0 \leq d < b$ and $m > 0$. 
\end{theorem}

\begin{theorem}
If $b$ is odd and $n$ is even, then $f_b(d,m,n)$ is composite for all $0\leq d< b$ and $m>0$.
\end{theorem}
\begin{proof}
Using the notation in Definition \ref{def:fb}, we have for $m> 0$, 
$$dR_b(m)b^{t+m}+dR_b(m) = dR_b(m)\left(b^{t+m}+1\right) > 2$$ which is even since $b^{t+m}+1$ is even.  Therefore $f_b(d,m,n)>2$ is even and thus composite.
\end{proof}

Using the algorithm in Section \ref{sec:additional} adapted to arbitrary bases, we found  various values of $n$ which are not covered by Theorem \ref{thm:baseb} such that $f_b(d,m,n)$ is composite.
Some of these values are 
listed in Table \ref{tbl:composite}.

\begin{table}[htbp] 
\begin{center}
\begin{tabular}{|c|l|}
\hline
base $b$ & Partial list of $n$ such that $f_b(d,m,n)$ is composite for all $0\leq d< b$ and $m>0$  \\
\hline\hline
$2$ & $2040, 8177, 18179, 32739, 71357, 71532, 101895, 131015, 318929, 403410, 404859$\\
$4$ &  $30, 72, 465, 1020, 3252, 4110, 4965, 4992, 5112, 5475, 6330, 6357, 6477, 6840, 7695$ \\
$6$ & $47215, 74090, 87110, 93870, 120745, 140525, 167400, 178500, 187180, 214055, 232545$ \\
$8$ &  $255, 315, 4305, 5670, 7035, 8400, 9180, 9765, 11130, 12495, 13860, 15225, 16590$ \\
$12$ & $1155, 23695, 28875, 30800, 41965, 53130, 64295, 75460, 86625, 97790, 108955, 120120$\\
$14$ & $46917,9151272, 9382542, 11892387, 14402232, 16912077, 17638335, 19421922$ \\
\hline

\end{tabular}
\end{center}
\caption{Some values of $n$ and $b$ such that $f_b(d,m,n)$ is composite for all $0\leq d< b$ and $m>0$ and were not covered by Theorem \ref{thm:baseb}.}
\label{tbl:composite}
\end{table}

\end{document}